\documentclass[12pt]{article}

\usepackage{amssymb,amsmath}
\usepackage{epsfig}
\usepackage{amsfonts, amsmath, amsthm, bbm, subcaption}
\usepackage{graphicx}
\usepackage{appendix,cite}
\usepackage{tikz, tikz-cd}
\usepackage{enumerate}
\usetikzlibrary{arrows}
\usepackage{pdfpages}
\usepackage{graphics}

\newtheorem{theorem}{Theorem}[section]

\theoremstyle{definition}

\theoremstyle{remark}

\newcommand{\Rel}[1]{\mbox{\mbox{Rel}$\left( #1;q \right)$}}
\newcommand{\Relp}[1]{\mbox{\mbox{Rel}$\left( #1;p \right)$}}
\newcommand{\Relpnum}[2]{\mbox{\mbox{Rel}$\left( #1;#2 \right)$}}
\newcommand{\spRel}[1]{\mbox{\mbox{spRel}$(#1;p)$}}

\newcommand{\TRelp}[3]{\mbox{\mbox{Rel}$_{#2,#3}(#1;p)$}}
\newcommand{\TRelpnum}[4]{\mbox{\mbox{Rel}$_{#2,#3}(#1;#4)$}}

\newtheorem{prob}{Problem}[section]

\begin{document}

\title{Roots of Two-Terminal Reliability}
\author{Jason Brown and Corey D. C. DeGagn\'e \\ Dalhousie University}

\maketitle

%
%

\begin{abstract}
Assume that the vertices of a graph $G$ are always operational, but the edges of $G$ are operational independently with probability $p \in[0,1]$. For fixed vertices $s$ and $t$, the \emph{two-terminal reliability} of $G$ is the probability that the operational subgraph contains an $(s,t)$-path, while the \emph{all-terminal reliability} of $G$ is the probability that the operational subgraph contains a spanning tree. Both reliabilities are polynomials in $p$, and have very similar behaviour in many respects. However, unlike all-terminal reliability, little is known about the roots of two-reliability polynomials.  In a variety of ways, we shall show that the nature and location of the roots of two-terminal reliability polynomials have significantly different properties than those held by roots of the all-terminal reliability.\\

\end{abstract}

\section{Introduction}
There are a variety of probabilistic models of network robustness, but the two most common in the literature are  \textbf{all-terminal reliability} and \textbf{two-terminal reliability}. Suppose that $G$ is a finite, undirected graph on vertex set $V$ (we shall assume, unless otherwise noted, all graphs under discussion are connected). Suppose further that each edge is independently operational (or ``up'') with probability $p \in [0,1]$.  The \textbf{all-terminal reliability} of $G$, $\Relp{G}$, is the probability that in the spanning subgraph of operational edges, all vertices can communicate (this is equivalent to the spanning subgraph of operational edges containing a spanning tree of $G$). If $s$ and $t$ are vertices of $G$, then the  \textbf{two-terminal reliability} of $G$ (or more precisely, of $(G,s,t)$), $\TRelp{G}{s}{t}$, is the probability that in the spanning subgraph of operational edges, $s$ and $t$ can communicate (this is equivalent to the spanning subgraph of operational edges containing an $s$-$t$ path). Both all-terminal and two-terminal reliabilities on graphs are instances of a more general \textbf{$K$-terminal reliability} of a graph, where a set of $K$ vertices are chosen and one is interested in the probability that all of those vertices can communicate with one another; the all-terminal and two-terminal reliabilities are the two extremes of the problem (where $K$ is either all the vertices or just two of them).

The research into all-terminal reliability and two-terminal reliability is vast (see, for example, \cite{colbook}), and what is striking is how many results for one has a similar analogue for the other. Both are well known to be polynomial functions with the same general shape in the interval:
\begin{itemize}
\item They are increasing on the interval $[0,1]$.
\item They cross the line $y = p$ at most once as long as the function is different from $p$. 
\item The functions are, in general, {\em $S$-shaped}, that is they are less than $p$ in some neighborhood of $0$, greater than $p$ in some neighborhood of $1$, and there is exactly one fixed point (this holds for all-terminal reliability whenever the graph is connected, has at least two vertices  and is not an edge, while for two-terminal reliability, this holds provided either (i) there is an edge between $s$ and $t$ and $G - st$ contains an $s$-$t$ path, or (ii) there is no edge between $s$ and $t$ and there is no edge whose removal separates $s$ and $t$). 
\end{itemize}
Both all-terminal and two-terminal reliabilities have underlying \textbf{coherent systems}, that is, collections of subsets of a set (namely the operational edge sets of the graph) that are closed upwards under containment.

The research for both all-terminal and two-terminal reliabilities have followed parallel tracks. Both can be calculated via \textbf{Factor theorems}, which are a recursion based on Bayes' Theorem:
\[ \Relp{G} = p\cdot \Relp{G\bullet e} + (1-p) \cdot \Relp{G - e}\]
\[ \TRelp{G}{s}{t} = p\cdot \TRelp{G\bullet e}{s}{t} + (1-p) \cdot \TRelp{G - e}{s}{t}\]
(Here $\bullet$ and $-$ denote the contraction and deletion of the edge $e$ in $G$, respectively.) Classes of graphs (complete graphs, series-parallel graphs, and so on) for which formulas are known for one of the two are almost always known for the other, with similar formulas. The problems of determining all-terminal and two-terminal reliabilities exactly are both known to be intractable ($\#P$-complete), and many of the same techniques (such as Sperner and Kruskal-Katona bounds) have been utilized to produce efficient algorithms to estimate the functions.

Given that all-terminal reliability is always a polynomial, one could naturally ask questions about the nature and location of the roots, and there have been a number of papers along these lines over the past three decades \cite{brownColbournRoots,brownColbournLogConcave,BM1,royleSokal,wagner}. The \textbf{all-terminal roots} (the roots of all-terminal reliability polynomials) have shown considerable structure, especially with respect to the unit disk $D_{1}(1) \equiv \{z \in \mathbb{C}: |z-1| \leq 1$\}, where they were conjectured in 1992 to be (and only disproved by the smallest of margins in \cite{royleSokal}).
\begin{itemize}
\item The real all-terminal roots are in $\{0\} \cup (1,2]$. \cite{brownColbournRoots}
\item The closure of the all-terminal roots includes the disk $D_{1}(1)$ \cite{brownColbournRoots}.  (While there are some roots outside the disk \cite{royleSokal}, it is not known whether the closure of the all-terminal roots contain any set of positive measure outside the disk. Currently \cite{BM1}, the root of furthest away from $z = 1$ has distance approximately $1.13$.)
\item Every (connected) graph has a subdivision whose roots are all real.  \cite{brownColbournLogConcave}
\item Every (connected) graph has a subdivision whose roots lie in the disk $D_{1}(1)$. \cite{brownColbournLogConcave}
\end{itemize} 

\begin{figure}
    \centering
    \includegraphics[width=3.8in]{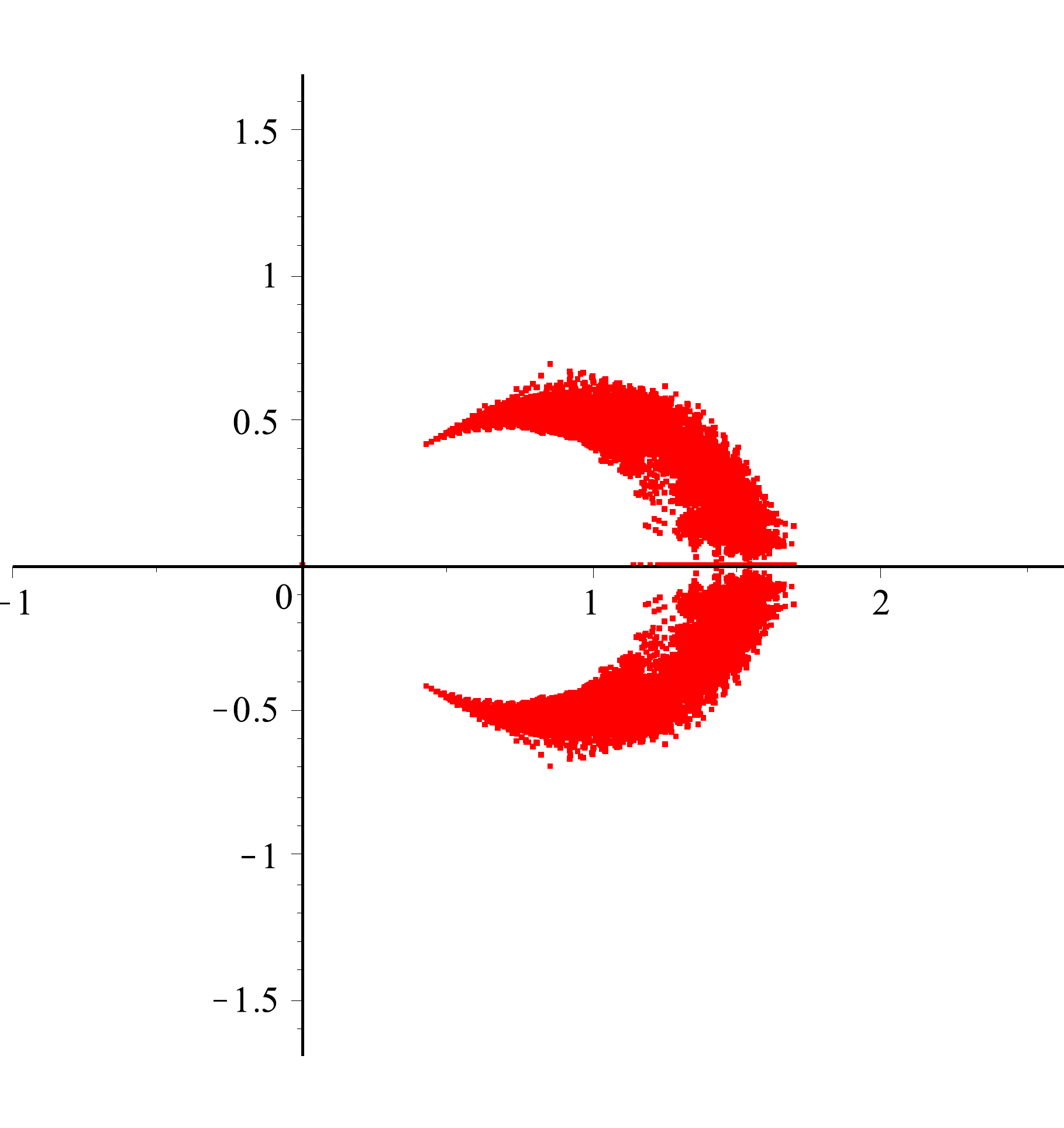}
    \caption{Plot of all-terminal roots of graphs of order $8$ (that is, with $8$ vertices).}
    \label{fig:All-Term-Roots-Order8}
\end{figure}

Surprisingly, very little is known about \textbf{two-terminal roots}, the roots of two-terminal reliability polynomials. Tungay \cite{T1,T2,T3,T4} provided a few families of two-terminal reliabilities, with all of the roots lying in the box 
\[ \{z \in \mathbb{C} : -0.8 < \Re(z) < 1.6,~ -0.65  < \Im(z) < 0.65 \}.\]
In light of how much all-terminal and two-reliabilities have in common, one might expect that their roots share similar features.
However, a plot of the roots of two-terminal reliabilities for simple graphs of small order (see Figure~\ref{fig:Two-Term-Roots-Order7}) show much different structure than that for all-terminal reliability (see Figure~\ref{fig:All-Term-Roots-Order8}), and in this paper we will delve into some of the ways in which they do indeed differ significantly.
\begin{figure}
    \centering
    \includegraphics[width=4.0in]{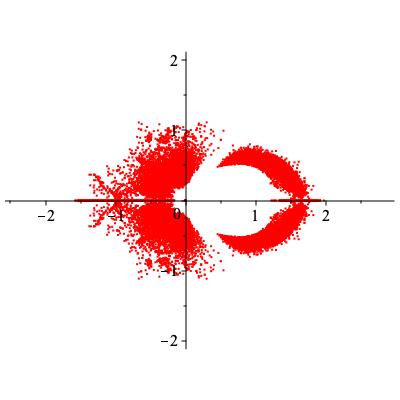}
    \caption{Plot of two-terminal roots of graphs of order $7$}
    \label{fig:Two-Term-Roots-Order7}
\end{figure}

We remark that our graphs allow for multiple edges, but not loops (as loops do not affect either  the all-terminal or two-terminal reliability). The \textbf{order} and \textbf{size} of a a graph denote the cardinality of the set of vertices and edges, respectively.

\section{Real Two-Terminal Roots}

One of the fundamental results on all-terminal roots is that every graph has a subdivision whose all-terminal roots are all real.  Such is not the case for two-terminal reliability. For example,  note that every subdivision of a cycle is again a cycle, so it suffices to consider  the real roots of cycles.  Note that if $s$ and $t$ are two vertices of the cycle $C_n$ at distance $k$, then
\[ \TRelp{C_n}{s}{t} = p^k+p^{n-k}-p^n,\]
as $s$ and $t$ can communicate if and only if at least one  of the two paths (of lengths $k$ and $n-k$) are operational. 
The well-known {\em Descartes' Rule of Signs} states that for a real polynomial, the number of positive roots (counting multiplicities) is at most the number of sign changes in the coefficients. As $\TRelp{C_n}{s}{t}$ has exactly one sign change, it has at most one positive root. By considering $\TRelpnum{C_n}{s}{t}{-p}$, one can see that $\TRelp{C_n}{s}{t}$ has at most $2$ negative roots. It follows that $\TRelp{C_n}{s}{t}$ has at most $k + 1 + 2 = k+3$  real roots. As $k \leq n/2$, we see that no matter what choice of $s$ and $t$ we make, it follows that no subdivision of a cycle of order at least $7$ has all real two-terminal roots (in fact, by direct calculations, one can replace $7$ by $5$, and even $C_4$, with adjacent terminals, has no subdivision for which the two-terminal roots are all real).

As noted earlier, the real all-terminal roots all lie in the set $\{0\} \cup (1,2]$. However, from Figure~\ref{fig:Two-Term-Roots-Order7}), it appears that there are negative two-terminal roots. Indeed this is the case. Consider the theta graph $\Theta_{l[k]}$, consisting of two vertices, $s$ and $t$, joined by $k$ internally disjoint paths, each of length $l$. Its two-terminal reliability is given by
\begin{equation}
	\label{eqn:TwoTerm-Real-GenTheta}
	\TRelp{\Theta^{l[k]}}{s}{t} = 1 - (1 - p^l)^k,
\end{equation}
as $s$ and $t$ can communicate if and only if at least one of the $k$ paths of length $l$ is operational.
If $l$ is even, we have $\TRelpnum{\Theta_{l[k]}}{s}{t}{-1} = 1 > 0$.  Now, for $\varepsilon > 0 $, what is the sign of $\TRelpnum{\Theta_{l[k]}}{s}{t}{-1 - \varepsilon}$? 
Let us choose $l$ even and large enough so that $\epsilon > 2^{\frac{1}{l}} - 1$. Then
$$1 - (-1 - \epsilon)^l  = 1 - (1 + \epsilon)^l < -1.$$
Then for any even $k$, $\TRelpnum{\Theta_{l[k]}}{s}{t}{-1 - \varepsilon} < 0$, and so by 
the Intermediate Value Theorem, there must be a real root in the interval $(-1 - \varepsilon, -1)$.

There can be roots even further to the left of $-1$. For example, for the four-cycle $C_{4}$ with nonadjacent terminals $s$ and $t$,
\[ \TRelp{C_4}{s}{t}  = 2p^{2}-p^{4} = p^{2}(2-p^{2}).\]
This places a two-terminal root at $-\sqrt{2} \approx -1.41$ (and at $\sqrt{2}$).

We can even move the roots further out to the left, but to do so, we'll need to introduce a general type of graph operation (see \cite{BM1,cox}). Suppose that $H$ is a graph with two distinct distinguished vertices $u$ and $v$ (we'll refer to $(H,u,v)$ as a \textbf{gadget}). Then for any graph $G$, an \textbf{edge substitution} of $H(u,v)$ into $G,$ denoted by $G[H(u,v)],$ is any graph formed by replacing each edge $\{x,y\}\in E(G)$ by a copy $H_{\{x,y\}}$ of $H,$ identifying vertices $u$ with $x$ and $v$ with $y$  (we assume that $G$ and all of the copies of $H$ are disjoint). We also call $G[H(u,v)]$ a \textbf{gadget replacement} of $G$ with $H$. An example is illustrated in Figure~\ref{fig:twoterm-root-gadgetreplacement}. 

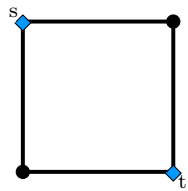
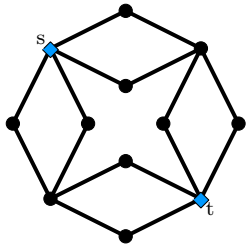
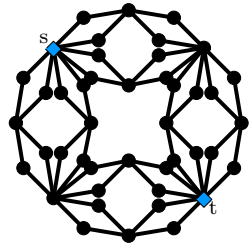
\begin{figure}
	\centering
	\definecolor{qqzzff}{rgb}{0,0.6,1}
	\begin{minipage}{0.45\linewidth}
	\begin{tikzpicture}[line cap=round,line join=round,>=triangle 45,x=.5cm,y=.5cm]
\clip(-5.2743904128230197,-4.249041386132193) rectangle (2.7320571278323365,2.9366380869399595);
\draw [line width=1.5pt] (-2,2)-- (2,2);
\draw [line width=1.5pt] (2,-2)-- (2,2);
\draw [line width=1.5pt] (2,-2)-- (-2,-2);
\draw [line width=1.5pt] (-2,-2)-- (-2,2);
\begin{scriptsize}
\draw [fill=qqzzff] (-2,2) ++(-2.5pt,0 pt) -- ++(2.5pt,2.5pt)--++(3pt,-3pt)--++(-3pt,-3pt)--++(-3pt,3pt);
\draw[color=black] (-2.25,2.25) node {s};
\draw[color=black] (2.25,-2.25) node {t};
\draw [fill=black] (2,2) circle (2.5pt);
\draw [fill=black] (-2,-2) circle (2.5pt);
\draw [fill=qqzzff] (2,-2) ++(-2.5pt,0 pt) -- ++(2.5pt,2.5pt)--++(3pt,-3pt)--++(-3pt,-3pt)--++(-3pt,3pt);
\end{scriptsize}
\end{tikzpicture}
	\subcaption{Initial graph $C_4$ with antipodal terminals $s$ and $t$.}
	\end{minipage} \\
	\begin{minipage}{0.45\linewidth}
	\begin{tikzpicture}[line cap=round,line join=round,>=triangle 45,x=.5cm,y=.5cm]
\clip(-4.348056695176196,-3.508556258653804) rectangle (5.67274717622092,3.654603923181391);
\draw [line width=1.5pt] (-2,2)-- (0,3);
\draw [line width=1.5pt] (0,3)-- (2,2);
\draw [line width=1.5pt] (2,2)-- (0,1);
\draw [line width=1.5pt] (0,1)-- (-2,2);
\draw [line width=1.5pt] (2,2)-- (1,0);
\draw [line width=1.5pt] (1,0)-- (2,-2);
\draw [line width=1.5pt] (2,-2)-- (3,0);
\draw [line width=1.5pt] (3,0)-- (2,2);
\draw [line width=1.5pt] (-2,-2)-- (0,-1);
\draw [line width=1.5pt] (0,-1)-- (2,-2);
\draw [line width=1.5pt] (2,-2)-- (0,-3);
\draw [line width=1.5pt] (0,-3)-- (-2,-2);
\draw [line width=1.5pt] (-2,-2)-- (-1,0);
\draw [line width=1.5pt] (-1,0)-- (-2,2);
\draw [line width=1.5pt] (-2,2)-- (-3,0);
\draw [line width=1.5pt] (-3,0)-- (-2,-2);
\begin{scriptsize}
\draw [fill=qqzzff] (-2,2) ++(-2.5pt,0 pt) -- ++(2.5pt,2.5pt)--++(3pt,-3pt)--++(-3pt,-3pt)--++(-3pt,3pt);
\draw [fill=black] (2,2) circle (2.5pt);
\draw [fill=qqzzff] (2,-2) ++(-2.5pt,0 pt) -- ++(2.5pt,2.5pt)--++(3pt,-3pt)--++(-3pt,-3pt)--++(-3pt,3pt);
\draw [fill=black] (-2,-2) circle (2.5pt);
\draw [fill=black] (0,3) circle (2.5pt);
\draw [fill=black] (0,1) circle (2.5pt);
\draw [fill=black] (1,0) circle (2.5pt);
\draw [fill=black] (3,0) circle (2.5pt);
\draw [fill=black] (0,-1) circle (2.5pt);
\draw [fill=black] (0,-3) circle (2.5pt);
\draw [fill=black] (-1,0) circle (2.5pt);
\draw [fill=black] (-3,0) circle (2.5pt);
\draw[color=black] (-2.25,2.25) node {s};
\draw[color=black] (2.25,-2.25) node {t};
\end{scriptsize}
\end{tikzpicture}
	\subcaption{First iteration of the gadget replacement, $C_4[C_4(s,t)]$}
	\end{minipage}\hfill
	\begin{minipage}{0.45\linewidth}
	\begin{tikzpicture}[line cap=round,line join=round,>=triangle 45,x=.5cm,y=.5cm]
\clip(-4.266767596500207,-3.6453027753876723) rectangle (5.792467427287507,3.545329104734388);
\draw [line width=1.5pt] (-2,2)-- (-0.8,2.2);
\draw [line width=1.5pt] (-0.8,2.2)-- (0,3);
\draw [line width=1.5pt] (0,3)-- (0.8,2.2);
\draw [line width=1.5pt] (0.8,2.2)-- (2,2);
\draw [line width=1.5pt] (2,2)-- (0.8,1.8);
\draw [line width=1.5pt] (0.8,1.8)-- (0,1);
\draw [line width=1.5pt] (0,1)-- (1,1.2);
\draw [line width=1.5pt] (1,1.2)-- (2,2);
\draw [line width=1.5pt] (-2,2)-- (-0.8,1.8);
\draw [line width=1.5pt] (-0.8,1.8)-- (0,1);
\draw [line width=1.5pt] (0,1)-- (-1,1.2);
\draw [line width=1.5pt] (-1,1.2)-- (-2,2);
\draw [line width=1.5pt] (-2,2)-- (-1.2,2.8);
\draw [line width=1.5pt] (-1.2,2.8)-- (0,3);
\draw [line width=1.5pt] (0,3)-- (1.2,2.8);
\draw [line width=1.5pt] (1.2,2.8)-- (2,2);
\draw [line width=1.5pt] (0,-1)-- (-1,-1.2);
\draw [line width=1.5pt] (0,-1)-- (1,-1.2);
\draw [line width=1.5pt] (-1,-1.2)-- (-2,-2);
\draw [line width=1.5pt] (1,-1.2)-- (2,-2);
\draw [line width=1.5pt] (0,-1)-- (0.8,-1.8);
\draw [line width=1.5pt] (0.8,-1.8)-- (2,-2);
\draw [line width=1.5pt] (-2,-2)-- (-0.8,-1.8);
\draw [line width=1.5pt] (-0.8,-1.8)-- (0,-1);
\draw [line width=1.5pt] (-2,-2)-- (-0.8,-2.2);
\draw [line width=1.5pt] (-0.8,-2.2)-- (0,-3);
\draw [line width=1.5pt] (0,-3)-- (0.8,-2.2);
\draw [line width=1.5pt] (0.8,-2.2)-- (2,-2);
\draw [line width=1.5pt] (-2,-2)-- (-1.2,-2.8);
\draw [line width=1.5pt] (-1.2,-2.8)-- (0,-3);
\draw [line width=1.5pt] (0,-3)-- (1.2,-2.8);
\draw [line width=1.5pt] (1.2,-2.8)-- (2,-2);
\draw [line width=1.5pt] (-1,0)-- (-1.2,1);
\draw [line width=1.5pt] (-1.2,1)-- (-2,2);
\draw [line width=1.5pt] (-1,0)-- (-1.8,0.8);
\draw [line width=1.5pt] (-1.8,0.8)-- (-2,2);
\draw [line width=1.5pt] (1,0)-- (1.2,1);
\draw [line width=1.5pt] (1.2,1)-- (2,2);
\draw [line width=1.5pt] (1,0)-- (1.8,0.8);
\draw [line width=1.5pt] (1.8,0.8)-- (2,2);
\draw [line width=1.5pt] (-2,-2)-- (-1.2,-1);
\draw [line width=1.5pt] (-1.2,-1)-- (-1,0);
\draw [line width=1.5pt] (2,-2)-- (1.2,-1);
\draw [line width=1.5pt] (1.2,-1)-- (1,0);
\draw [line width=1.5pt] (-1,0)-- (-1.8,-0.8);
\draw [line width=1.5pt] (-1.8,-0.8)-- (-2,-2);
\draw [line width=1.5pt] (1,0)-- (1.8,-0.8);
\draw [line width=1.5pt] (1.8,-0.8)-- (2,-2);
\draw [line width=1.5pt] (-3,0)-- (-2.2,0.8);
\draw [line width=1.5pt] (-2.2,0.8)-- (-2,2);
\draw [line width=1.5pt] (-2,2)-- (-2.8,1.2);
\draw [line width=1.5pt] (-2.8,1.2)-- (-3,0);
\draw [line width=1.5pt] (-3,0)-- (-2.2,-0.8);
\draw [line width=1.5pt] (-2.2,-0.8)-- (-2,-2);
\draw [line width=1.5pt] (-2,-2)-- (-2.8,-1.2);
\draw [line width=1.5pt] (-2.8,-1.2)-- (-3,0);
\draw [line width=1.5pt] (2,2)-- (2.2,0.8);
\draw [line width=1.5pt] (2.2,0.8)-- (3,0);
\draw [line width=1.5pt] (3,0)-- (2.2,-0.8);
\draw [line width=1.5pt] (2.2,-0.8)-- (2,-2);
\draw [line width=1.5pt] (2,2)-- (2.8,1.2);
\draw [line width=1.5pt] (2.8,1.2)-- (3,0);
\draw [line width=1.5pt] (2,-2)-- (2.8,-1.2);
\draw [line width=1.5pt] (2.8,-1.2)-- (3,0);
\begin{scriptsize}
\draw [fill=qqzzff] (-2,2) ++(-2.5pt,0 pt) -- ++(2.5pt,2.5pt)--++(3pt,-3pt)--++(-3pt,-3pt)--++(-3pt,3pt);
\draw [fill=black] (2,2) circle (2.5pt);
\draw [fill=qqzzff] (2,-2) ++(-2.5pt,0 pt) -- ++(2.5pt,2.5pt)--++(3pt,-3pt)--++(-3pt,-3pt)--++(-3pt,3pt);
\draw [fill=black] (-2,-2) circle (2.5pt);
\draw [fill=black] (0,3) circle (2.5pt);
\draw [fill=black] (0,1) circle (2.5pt);
\draw [fill=black] (1,0) circle (2.5pt);
\draw [fill=black] (3,0) circle (2.5pt);
\draw [fill=black] (0,-1) circle (2.5pt);
\draw [fill=black] (0,-3) circle (2.5pt);
\draw [fill=black] (-1,0) circle (2.5pt);
\draw [fill=black] (-3,0) circle (2.5pt);
\draw [fill=black] (-0.8,2.2) circle (2.5pt);
\draw [fill=black] (0.8,2.2) circle (2.5pt);
\draw [fill=black] (0.8,1.8) circle (2.5pt);
\draw [fill=black] (1,1.2) circle (2.5pt);
\draw [fill=black] (-0.8,1.8) circle (2.5pt);
\draw [fill=black] (-1,1.2) circle (2.5pt);
\draw [fill=black] (-1.2,2.8) circle (2.5pt);
\draw [fill=black] (1.2,2.8) circle (2.5pt);
\draw [fill=black] (-1,-1.2) circle (2.5pt);
\draw [fill=black] (1,-1.2) circle (2.5pt);
\draw [fill=black] (0.8,-1.8) circle (2.5pt);
\draw [fill=black] (-0.8,-1.8) circle (2.5pt);
\draw [fill=black] (-0.8,-2.2) circle (2.5pt);
\draw [fill=black] (0.8,-2.2) circle (2.5pt);
\draw [fill=black] (-1.2,-2.8) circle (2.5pt);
\draw [fill=black] (1.2,-2.8) circle (2.5pt);
\draw [fill=black] (-1.2,1) circle (2.5pt);
\draw [fill=black] (-1.8,0.8) circle (2.5pt);
\draw [fill=black] (1.2,1) circle (2.5pt);
\draw [fill=black] (1.8,0.8) circle (2.5pt);
\draw [fill=black] (-1.2,-1) circle (2.5pt);
\draw [fill=black] (1.2,-1) circle (2.5pt);
\draw [fill=black] (-1.8,-0.8) circle (2.5pt);
\draw [fill=black] (1.8,-0.8) circle (2.5pt);
\draw [fill=black] (-2.2,0.8) circle (2.5pt);
\draw [fill=black] (-2.8,1.2) circle (2.5pt);
\draw [fill=black] (-2.2,-0.8) circle (2.5pt);
\draw [fill=black] (-2.8,-1.2) circle (2.5pt);
\draw [fill=black] (2.2,0.8) circle (2.5pt);
\draw [fill=black] (2.2,-0.8) circle (2.5pt);
\draw [fill=black] (2.8,1.2) circle (2.5pt);
\draw [fill=black] (2.8,-1.2) circle (2.5pt);
\draw[color=black] (-2.25,2.25) node {s};
\draw[color=black] (2.25,-2.25) node {t};
\end{scriptsize}
\end{tikzpicture}
	\subcaption{Second iteration of the gadget replacement.}
	\end{minipage}
	\caption{Example of a sequence of gadget replacements $G[H]$ with $G$ and $H$ both copies of $C_4$.}
	\label{fig:twoterm-root-gadgetreplacement}
\end{figure}

It can be shown \cite{BM1,cox} that 
\begin{eqnarray*} 
\Rel{G[H(u,v)]} & = & (\Rel{H} + \spRel{H,u,v})^m \cdot \\
 & & \Relpnum{G} {\frac{\Rel{H}}{\Rel{H} + \spRel{H,u,v)}}},
 \end{eqnarray*}

However, the gadget replacement formula for two-terminal reliability is much simpler.

\begin{theorem}
\label{gadgettheorem}
Let $H$ be a graph with two distinct distinguished vertices $u$ and $v$. Then for any graph $G$ with terminals $s$ and $t$,
\begin{eqnarray} 
\label{gadgetformula}
\TRelp{G[H(u,v)]}{s}{t} & = & \TRelpnum{G}{s}{t}{\TRelp{H}{u}{v}}
\end{eqnarray}
\end{theorem}

\begin{proof}
If $F$ is a graph with edge set $E(F)$, and $S \subseteq E(F)$, $F(S)$ will denote the spanning subgraph of $F$ with edge set $S$. Let $\mathcal{S}_{G}$ be the collection of subsets of edges $S$ of $G$ such that $s$ and $t$ can communicate in $G(S)$. For a spanning subgraph $S$ of $G[H(u,v)]$, let $\widehat{S}$ be the set of edges $e$ of $G$ such that the endpoints of $e$, $u_{e}$ and $v_{e}$, can communicate in $H_{u_{e},v_{e}}$. Then it is straightforward to see that $s$ and $t$ can communicate in $G[H(u,v)](S)$ if and only if they can communicate in $G(\widehat{S})$, so it follows that 
\begin{eqnarray*} 
\TRelp{G[H(u,v)]}{s}{t} & = & \sum_{S \in \mathcal{S}_{G}} (\TRelp{H}{u}{v})^{i}(1-\TRelp{H}{u}{v})^{m-i}\\
                                    & = & \TRelpnum{G}{s}{t}{\TRelp{H}{u}{v}}
\end{eqnarray*}
\end{proof}

We can use Theorem~\ref{gadgettheorem} to find two-terminal roots even further to the left. Suppose we take $G = G_{0} = C_4$, with $s$ and $t$ being a pair of nonadjacent vertices, and consider the graph $G_{i+1} = G_{i}[G(s,t)]$, with terminals $s$ and $t$. From formula (\ref{gadgetformula}), we see that the two-terminal roots of $G_{i+1}$ are the solutions to the equation 
\begin{eqnarray} 
2p^2-p^4 & = & r, \label{C4recursiveroot}
\end{eqnarray}
where $r$ is a two-terminal root of $G_{i}$.
The solutions to (\ref{C4recursiveroot}) are
\[ p = \pm \sqrt{1 \pm r}.\]
Thus, setting $f(r) = -\sqrt{1-r}$, we find that the negative of the golden ratio, 
\[ -\varphi = -\frac{1+\sqrt{5}}{2} \approx -1.618034\]
is an attractive fixed point of $f$ (as $f(-\varphi) =  -\varphi$ and $f^\prime(-\varphi) = 0$), and hence there are two-terminal roots approaching $-\varphi$.

\vspace{0.25in}
The disk for which the all-terminal roots were originally conjectured to be in was $|z - 1| \leq 1$, and the furthest an all terminal root has been located from $z = 1$ is approximately $1.13$. In the two-terminal case, we have already seen the existence of roots much further away from $1$ (with distance approaching $1 + \varphi \approx 2.618034$). 

\section{Density of Two-Terminal Roots}

With regards to the all-terminal reliability, while not every root is in the closed unit disk centered at $z = 1$, the closure of the all-terminal roots is known to contain the closed unit disk, and is not known to contain any set of positive measure outside the disk.  What is true of two-terminal reliability roots? Might it be the case that there is no set of positive measure outside the unit disk at $z = 1$? The answer is emphatically no.

\begin{theorem}
The closure of two-terminal reliability roots contain the closed unit disks centred at $0$ and at $1$.
\label{thm:TwoTerm-ClosureDisks}
\end{theorem}
\begin{proof}
Let us consider again the two-terminal reliability polynomial of $\Theta^{l[k]}$, the generalized theta graph with $k$ paths of length $l$ joining $s$ and $t$: 
$$\TRelp{\Theta^{l[k]}}{s}{t} = 1-(1-p^l)^k.$$
The two-terminal roots of these graphs can be computed as follows:
\begin{align*}
    &\TRelp{\Theta^{l[k]}}{s}{t} = 0 \\
    \iff & (1-p^l)^k = 1\\
    \iff & 1-p^l = \omega \text{ (for $\omega$ some $k^{th}$ root of unity)} \\
    \iff & p^l = 1-\omega \\
    \iff &p =  \nu \text{ (for $\nu$ some $l^{th}$ root of $1-\omega$)}
\end{align*}
Let $r$ and $\theta$ satisfy $0 < r < 1$ and $0 < \theta < 2\pi$, and let $\epsilon > 0.$ Without loss of generality, $0 < r - \epsilon, r + \epsilon < 1, 0 < \theta - \epsilon,$ and $\theta + \epsilon < 2\pi$.  We will show that there is a value $\omega$, a $k$th root of unity, so that some $\nu$, an $l^{th}$ root of $1-\omega$, is in the small pie-shaped piece $\{Re^{i\gamma} | r - \epsilon < R < r + \epsilon \text{ and } \theta - \epsilon < \gamma < \theta + \epsilon\}$; this will show that the closure of the two-terminal roots of the generalized theta graphs contains the unit disk centered at the origin.

First, as the $l$ arguments of the $l^{th}$ roots of unity of a number are equally spaced out, for all sufficiently large $l \geq L$, we can ensure that the argument $\gamma$ of some $l^{th}$ root of any nonzero number is in $[\theta - \epsilon, \theta + \epsilon]$. 

Second, the $k^{th}$ roots of $1$, running over all $k$, fill up the boundary of the unit circle centered at $z = 0 $.  So, if we consider one minus these values, the resulting complex numbers take on values whose moduli are close to every number in $[0,2]$.
We can therefore choose a $k^{th}$ root of unity $\omega$ such that $(r - \epsilon)^L <  |1 - \omega| < (r + \epsilon)^L$. Then there is an $L^{th}$ root, say $\nu$, of $1-\omega$ such that $r - \epsilon < |\nu| < r + \epsilon$, and the argument $\gamma$ of $\nu$ lies in $[\theta - \epsilon, \theta + \epsilon]$. As noted earlier, this implies that the closure of the two-terminal reliabilities of $\Theta_{l[k]}$ contain the closed unit disk centered at $0$.

What about for the closed unit disk centred at $1$?  For any graph $G$ with terminals $s$ and $t$, replace every edge in $G$ by a bundle $B_{m}$ of size $m$ (denote the vertices of $H$ by $u$ and $v$).  Then by Theorem~\ref{gadgettheorem}, we find that 
\begin{eqnarray*} 
\TRelp{G[B_m(u,v)]}{x}{y} & = & \TRelpnum{G}{x}{y}{1-(1-p)^m}.
\end{eqnarray*}
Thus the two-terminal roots of $(G[B_m(u,v)],s,t)$ are one minus the $m$--th roots of $1-r$, where $r$ is any two-terminal root of $(G,s,t)$.
By a similar argument to the one above, the closure of the set 
\[ \{ z \in \mathbb:{C}: z^m = 1-r \mbox{~ for some $r$ a root of $\TRelp{\Theta^{l[k]}}{s}{t}$} \} \] contains the unit disk centered at $0$, and hence the two-terminal roots of $\TRelp{{\Theta^{l[k]}}[B_m(u,v)]}{s}{t}$, as $l$, $k$ and $m$ range over all positive integers, contains the unit disk centered at $1$ as well.
\end{proof}

%


One final related problem. It was shown in \cite{brownColbournRoots} that every graph has a subdivision whose roots lie in the unit disk centered at $1$ -- might this still be true for two-terminal reliability, even though the roots are certainly not contained in the disk? Again, the answer is no. For example, consider the cycle once again. Let the order of the cycle be $n \geq 3$ and choose terminals $s$ and $t$ at distance $k$ ($1 \leq k \leq n/2$). As previously noted,  
\[ \TRelp{C_n}{u}{v} = p^{k}+p^{n-k}-p^{n} = -p^{k}(p^{n-k} - p^{n-2k} - 1).\]
We will show that the polynomial $g(p) = p^{n-k} - p^{n-2k} - 1$ has a root in the open left-half plane, and hence no subdivision of $C_{n}$ for $n \geq 3$ has all its two-terminal roots in the disk $|z -1| \leq 1$.

The Hermite-Biehler Theorem (see, for example \cite{wagner}) states that a real polynomial 
\[ f(x) = \sum_{i=0}^{d} a_{i}x^{i}\]
with positive leading coefficient is \textbf{weakly stable}, that is, has all its roots in the closed left-half plane, if and only if $f_{\mbox{even}} = \sum a_{2i}x^{i}$  and $f_{\mbox{odd}} =\sum a_{2i+1}x^{i}$ both have positive leading coefficient (or are identically $0$), both have all nonpositive real roots, and the roots of $f_{\mbox{odd}}$ {\em interlace} those of $f_{\mbox{even}}$ (that is, if $s_{1} \leq s_{2} \leq \cdots s_{l}$ and $r_{1} \leq r_{2} \leq \cdots \leq r_{m}$ are the roots of $f_{\mbox{odd}}$ and $f_{\mbox{even}}$ respectively, then $l \leq m$ and $s_{1} \leq r_{1} \leq s_{2} \leq r_{2} \leq \cdots$). The upshot is that we can show that $g$ has a root in the open left-half plane if 
\[ f(p) = (-1)^{n-k} g(-p) =  p^{n-k} + (-1)^{k+1}p^{n-2k} + (-1)^{n-k+1}\]
has a root in the right-half plane, i.e. $f$ is not weakly stable. 

To show this, we consider four cases, based on the parities of $n$ and $k$:
\begin{itemize}
\item $n$ and $k$ both even: Here $f_{\mbox{even}} = p^{(n-k)/2} - p^{(n-2k)/2} -1$, which has a positive root (as $n > k$). By the Hermite-Biehler Theorem, $f$ is not weakly stable.
\item $n$ is even and $k$ is odd: Here $f_{\mbox{even}} = p^{(n-2k)/2} + 1$, which has non-real roots unless $n = 2k$, in which case it has no roots. However, in this case, $f_{\mbox{odd}} = p^{(n-k-1)/2}$, and the exponent is positive (as $k \geq 3$). By the Hermite-Biehler Theorem, $f$ is not weakly stable, since if it were, $f_{\mbox{even}}$ must be at least as as many roots as $f_{\mbox{odd}}$ does.
\item $n$ is odd and $k$ is even: Here $f_{\mbox{even}} = 1$, which has no roots,  while $f_{\mbox{odd}} = p^{(n-k-1)/2} + p^{(n-2k-1)/2}  = p^{(n-2k-1)/2} \left( p^{k/2}+1 \right)$, which has at least two roots. As in the previous case, $f$ is not weakly stable.
\item $n$ and $k$ both odd: Here $f_{\mbox{even}} = p^{(n-k)/2} -1$, which has a positive root (as $n > k$). Again, by the Hermite-Biehler Theorem, $f$ is not weakly stable.
\end{itemize}

Thus every cycle of length at least $3$ has a root in the left-half plane, outside the unit disk centered at $0$, and hence {\em every} subdivision of {\em any} such a graph has a root outside the disk.

\section {Two-Terminal Fractals}

We are going to make even more use out of Theorem~\ref{gadgettheorem} and gadgets.  
We know that for graphs $G$ and $H$ with terminals $s$ and $t$ and $u$ and $v$, respectively,
\begin{eqnarray*} 
\TRelp{G[H(u,v)]}{s}{t} & = & \TRelpnum{G}{s}{t}{\TRelp{H}{u}{v}}.
\end{eqnarray*}
It would be natural to choose $(H,u,v) = (G,s,t)$, and ask what happens to the roots, here and under repeated iterations (see, for example, Figure~\ref{fig:twoterm-root-gadgetreplacement}). For example, let's take the graph $G = C_{4}$ with $s$ and $t$ two adjacent vertices. The roots of the third and sixth iteration are shown in Figure~\ref{fig:C4-adjacentTerminalsRoots-Fractal}.

\begin{figure}[h!]
    \centering
    \begin{minipage}{0.45\linewidth}
    		\includegraphics[width=3.0in]{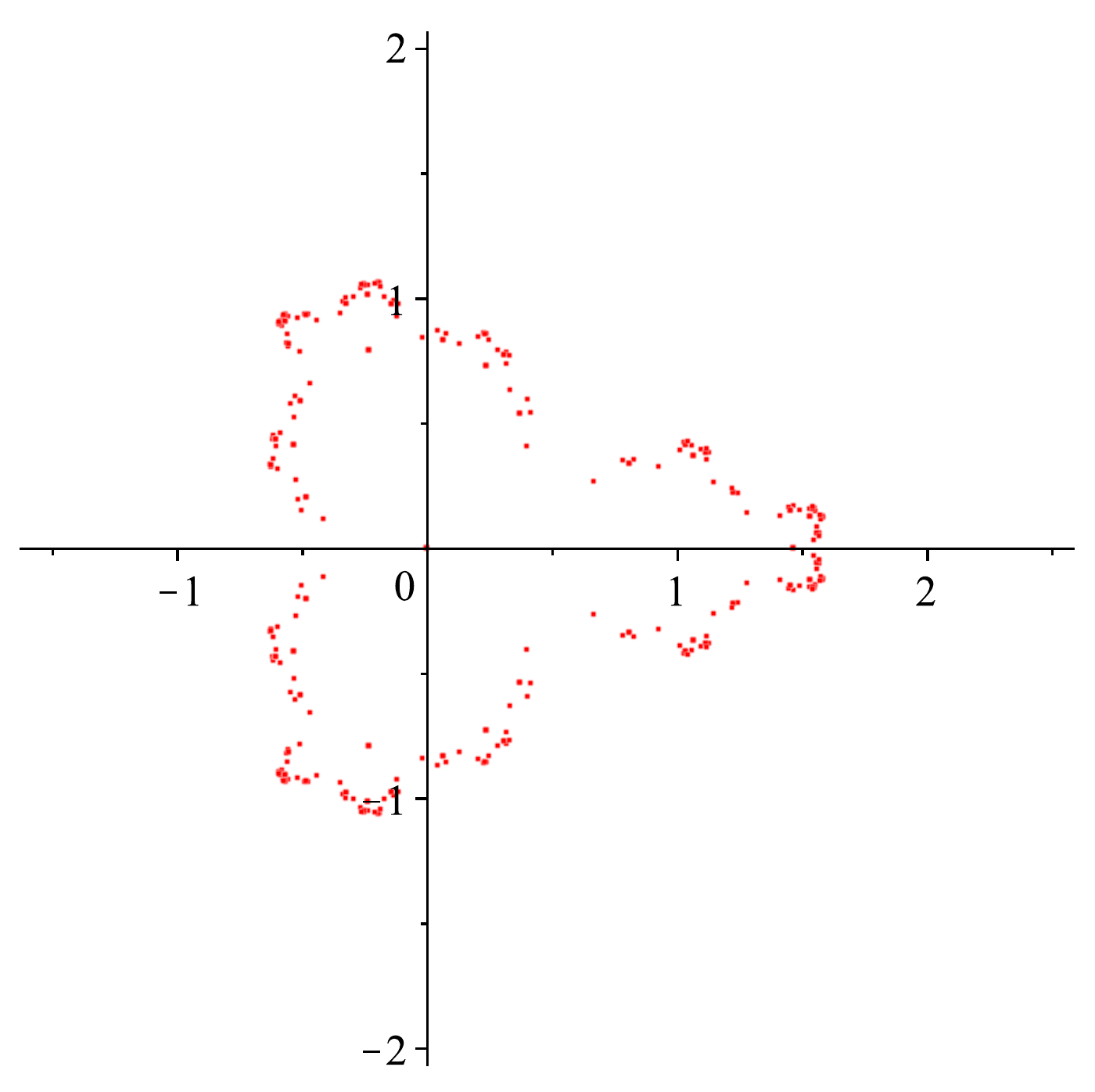}
    \subcaption{Plot of the two-terminal roots of $C_4[C_4[C_4]]$ with $s$ and $t$ adjacent in $C_4$.}
    \end{minipage}\hfill
    \begin{minipage}{0.45\linewidth}
    \includegraphics[width=3.0in]{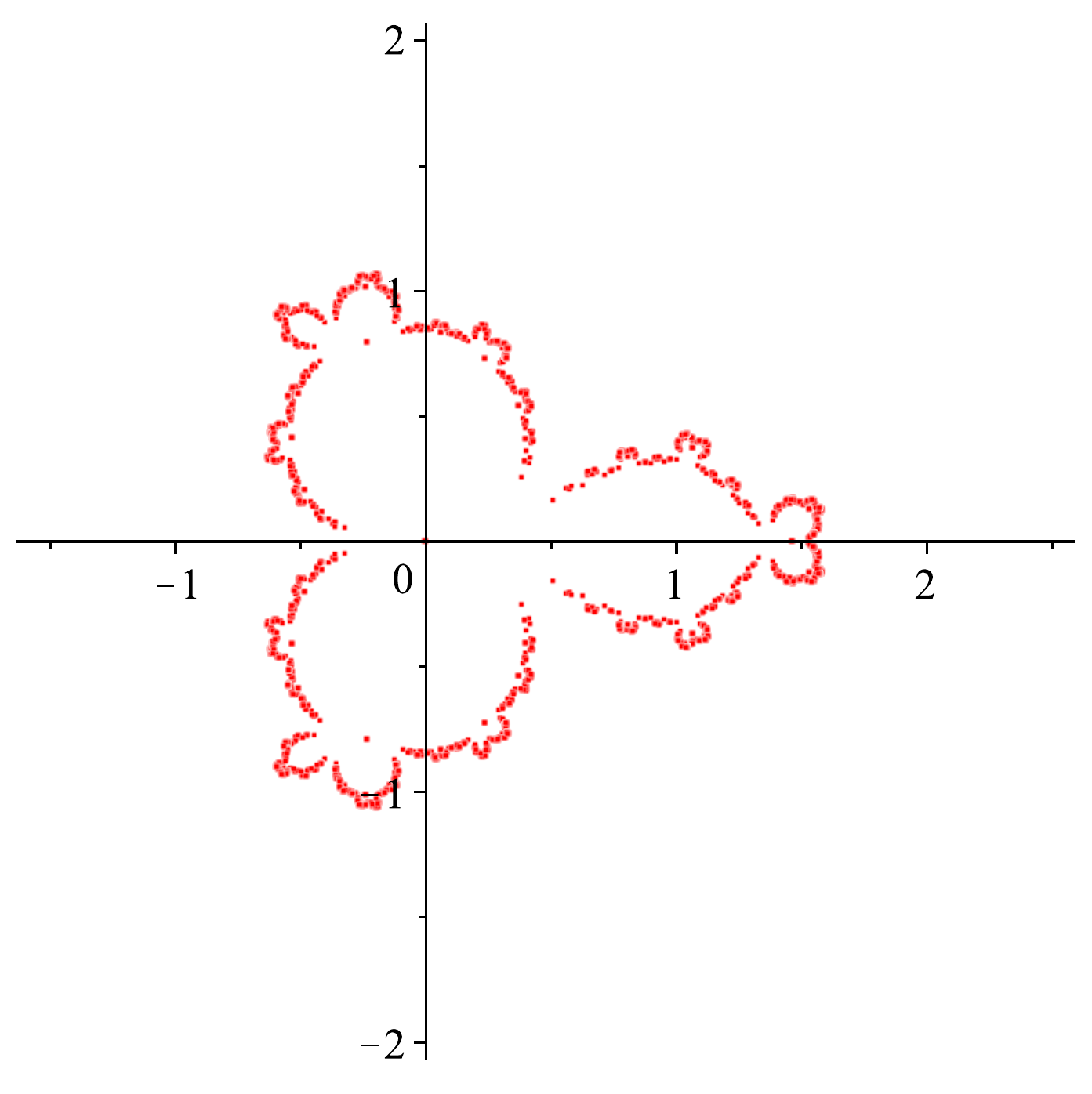}
    \subcaption{Plot of the two-terminal roots of $C_4[C_4[C_4[C_4[C_4]]]$ with $s$ and $t$ adjacent in $C_4$.}
    \end{minipage}
    \caption{Building a two-terminal attractor.}
    \label{fig:C4-adjacentTerminalsRoots-Fractal}
\end{figure}

If $G$ is a graph with (distinct) terminals $s$ and $t$, we set $G_1 = G$, and $G_{i} = G[G_{i-1}(s,t)]$ for $i \geq 2$. It follows immediately that if $f = \TRelp{G}{s}{t}$, then 
$\TRelp{G_{i}}{s}{t} = f^{\circ k}$ (the $k$-th composition of $f$ with itself). 
Moreover, the set of roots $\mathcal{T}_{i}$ of $\TRelp{G_{i}}{s}{t}$ are nested, i.e. 
$\mathcal{T}_{1}  \subseteq \mathcal{T}_{2} \subseteq  \mathcal{T}_{3} \cdots$) as $0$ is a root of $f = \TRelp{G}{s}{t}$.  

We define the \textbf{two-terminal attractor} of $(G,s,t)$ as 
\begin{eqnarray*} 
\mathcal{T}(G,s,t) & = & \overline{\bigcup_{k \geq 1} \{z \in {\mathbb C}: z \mbox{~ is a two-terminal root of } G_{k}\}}\\
 & = & \overline{\bigcup_{k \geq 1} f^{-\circ k}(\{0\})},
 \end{eqnarray*}
the closure of the \textbf{inverse orbit} of $0$ under $f$ (the inverse here denotes the set inverse: $f^{-1}(S) = \{ z : f(z) \in S\}$).

It is not hard to see that for any nonzero real polynomial $g$, the roots $g^{\circ k}$ are bounded. (The argument follows along the lines of Lemma~3.2.12 in \cite{hickmanPhDthesis}: We can assume that $g  = c_dx^d + c_{d-1}x^{d-1} + \cdots + c_1x+c_0$ has degree $d \geq 2$. If $C = \max \{ |c_{i}|: 0 \leq i \leq d-1\}$ and $R = \max  \left\{ (2/|c_d|)^{1/n-1}, C/|c_d| + 1\right\}$, then elementary inequalities show that $|z| > R$ imply that $|g(z) | > R$, from which we deduce that all roots of $g^{\circ k}$ are bounded by $R$). 
It follows that for any graph $G$ and terminals $s$ and $t$, the two-terminal attractor is bounded.

What is striking is the iterated gadget replacement operation seems to produce a fractal. As the following shows, this is not a coincidence. (In contrast, there is no known way to produce fractals for all-terminal reliability.) We refer the reader to \cite{beardon} for terminology and basic results concerning Julia and Fatou sets of rational functions. Our approach is modelled on that for {\em independence fractals} described in \cite{bhnindfractal,hickmanPhDthesis}.

\begin{theorem}
Let $G$ be a (connected) graph of order $n$ and size $m$, both at least $2$, with terminals $s$ and $t$, where we assume that every edge of $G$ is on some path between $s$ and $t$ (any edge not on an $s$-$t$ path can be removed without affecting the two-terminal reliability). Assume further that $G$ is not an $s$-$t$ path. Set $G_0 = G$, and $G_{i} = G[G_{i-1}(s,t)]$ for $i \geq 1$. Then 
\begin{enumerate}[i.]
\item if $s$ and $t$ are nonadjacent, then the two-terminal attractor $\mathcal{T}(G,s,t)$ is the \textbf{Julia set} of $\TRelp{G}{s}{t}$, and
\item if $s$ and $t$ are adjacent, then $\mathcal{T}(G,s,t)$ can be partitioned into the inverse orbits of $0$ under $\TRelp{G}{s}{t}$ and the Julia set of $\TRelp{G}{s}{t}$.
\end{enumerate}
\end{theorem}
\begin{proof}
Set $T = \TRelp{G}{s}{t}$ and $T_i = \TRelp{G_i}{s}{t}$, so $T_{i} = T^{\circ i}$. 
Now 
\begin{eqnarray*}
T & = & \TRelp{G}{s}{t}\\
  & = & \sum_{S \in {\mathcal S}_{G}} p^{|S|}(1-p)^{|E(G)-|S|}\\
  & = & \sum_{i \geq 1} {N_{i}} p^{i}(1-p)^{m-i} \\
  & = & a_{1}p + a_{2}p^2 + \cdots,
\end{eqnarray*}
where $m $ is the number of edges in $G$ and $a_{1} = N_{1}$ is the number of edges between $s$ and $t$. 

We begin assuming, for part (i), that $s$ and $t$ are adjacent, so $a_{1} \geq 1$ ($G$ may contain multiple edges, so $a_i$ may be greater than $1$). Clearly $0$ is a fixed point of $T = \TRelp{G}{s}{t}$. Moreover, as $T^\prime(0) = a_{i}$, it is either a \textbf{repelling} fixed point if $a_i > 1$ or a \textbf{rationally indifferent} fixed point if $a_i = 1$, but in either case, it follows (see Theorems 6.4.1 and 6.5.1 of \cite{beardon}) that $z = 0$ is in the Julia set of $T$. 

Note that the coefficient of $p$ in $T$ is $N_1$, so that if $T$ had degree $1$, then 
$T = N_{1}p$, so by the linear independence of the basis $\{ p^i(1-p)^{m-i-1}: i = 0,1,\ldots,m-1\}$ for the space of real polynomials of degree at most $m$, we must have $N_{i} = 0$ for all $i \neq 1$, a contradiction as $m \geq 2$ implies that $T = N_{1}p(1-p)^{m-1}$ has degree $m \geq 2$. Thus we know that $T$ has degree at least $2$. Furthermore, as $0$ is in the Julia set of $T$, by Theorem~4.2.7(ii) of \cite{beardon}, the closure of 
\[ \bigcup_{k \geq 1} T^{-k}(0) = \bigcup_{k \geq 1} \{z \in {\mathbb C}: z \mbox{~ is a two-terminal root of } G_{k}\}\]
is the Julia set $J(T)$ of $T$. This completes the proof of (i).

We move onto part (ii), by assuming now that $s$ and $t$ are not adjacent, so that $a_{1} = 0$.
Let $R = \mbox{Roots}(T)$ be the roots of $T$.  Here we find that $0$ is a (super)attracting fixed point of $T = \TRelp{G}{s}{t}$, as $\TRelp{G}{s}{t}$ and its derivative are $0$ when $p = 0$. It follows (see \cite[pg. 104]{beardon}) that $0$ is \underline{not} in the Julia set (but in the Fatou set) of $T$, and therefore the entire inverse orbit of $0$ is outside the Julia set of $T$. Let $k$ be the length of a shortest $s$-$t$ path (so $k \geq 2$); it should be clear that $T^\prime = T/p^{k}$ is a polynomial with the same nonzero roots as $T$. From the fact that $G$ is not an $s$-$t$ path, it follows that $\TRelp{G}{s}{t}$ is not a power of $p$, and hence the set $R^\prime = R - \{0\}$, is nonempty (and is the roots of $T^\prime$). Now for all $r \in R^\prime$, $r$ is not in a periodic point, as $T(r) = 0$ and $T(0) = 0$. Moreover, for the same reason, the \textbf{forward orbit} of $r$, $\{ T^{k}(r): k \geq 1\}$, is not dense on any curve, and from Theorem~3.2.10 of \cite{hickmanPhDthesis}, it follows that, in the Hausdorff metric, 
\[ \lim_{k \rightarrow \infty} T^{-k}(R^\prime) =  J(T),\] 
the Julia set of $T$. 
Finally, Julia sets are always \textbf{perfect}, that is equal to its accumulation points (c.f.\  Theorem 4.2.4. of \cite{beardon}), so finally we conclude that the Julia set of $T$ is equal to the accumulation points of $\cup \mbox{Roots}(T_i)$. None of the points of $\cup \mbox{Roots}(T_i)$ lie in the Julia set (as the Julia set is closed under backward and forward orbits), so we conclude that $\mathcal{T}(G,s,t)$ can be partitioned into the inverse orbits of $0$ under $\TRelp{G}{s}{t}$ and the Julia set of $\TRelp{G}{s}{t}$.

\end{proof}
\noindent (We left out the case when the graph $G$ is an $s$-$t$ path, but in this case, the two-terminal reliability is a power of $p$, and the two-terminal attractor is just $\{0\}$).

\vspace{0.25in}
So we see that while the Julia set of the two-terminal reliability shows up precisely when the terminals are adjacent, it is embedded in the attractor when the terminals are nonadjacent, as the accumulation points. Figure~\ref{fig:C4-antipodalTerminalsRoots-Fractal} shows (an approximation of) the two-terminal attractor for $C_4$ with nonadjacent terminals. (It is interesting to compare Figures~\ref{fig:C4-adjacentTerminalsRoots-Fractal} and \ref{fig:C4-antipodalTerminalsRoots-Fractal} and see how much a different choice of terminals makes to the two-terminal attractor!)

\begin{figure}
    \centering
    \includegraphics[width=3.65in]{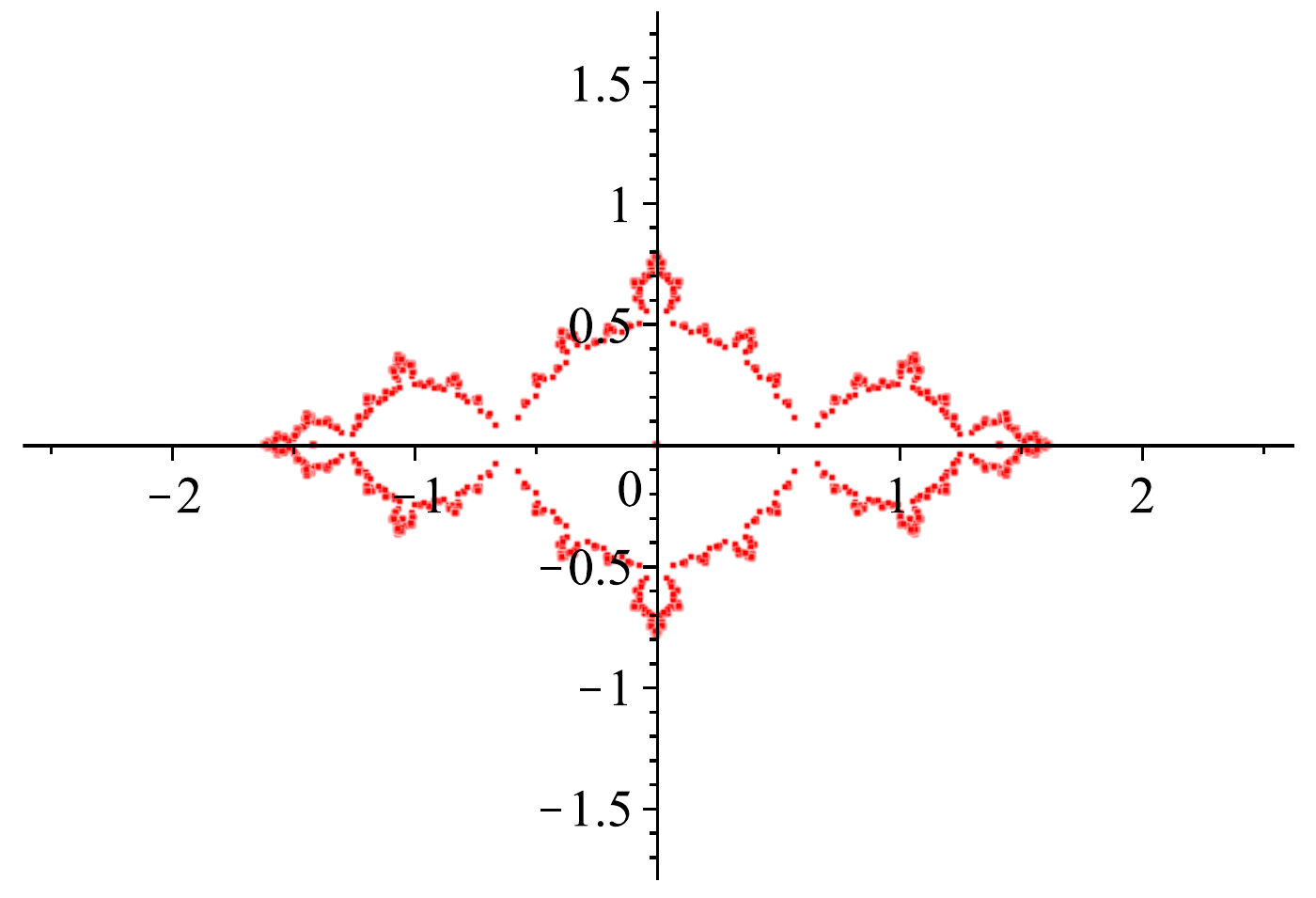}
    \caption{Plot of the union of two-terminal roots of repeated gadget replacement on $C_4$ with two nonadjacent terminals.}
    \label{fig:C4-antipodalTerminalsRoots-Fractal}
\end{figure}

In either case, it is worthwhile to investigate the Julia sets of two-terminal reliabilities (in most cases, Julia sets are \textbf{fractals}, exhibiting self-similiarity). One natural question to ask is whether the Julia set is connected -- as a topological space (which is different from asking whether the underlying graph is connected!). Here the problem appears in general to be difficult. It seems from Figures~\ref{fig:C4-adjacentTerminalsRoots-Fractal}(b) and \ref{fig:C4-antipodalTerminalsRoots-Fractal} that the Julia sets for $C_4$, whether the terminals are adjacent or not, are connected (although the resulting Julia sets are quite different!). A well known result (see, for example, \cite{beardon}) is that if $f$ is a polynomial of degree at least $2$, then $J(f)$ is connected if and only if the forward orbit of each of $f$'s critical points is bounded.

The first nontrivial graph to consider is $C_{3}$, with, of course, two adjacent terminals, $s$ and $t$. An approximation to its two terminal attractor is shown in Figure~\ref{fig:C36iterations}. 

\begin{figure}
    \centering
    \includegraphics[width=3.65in]{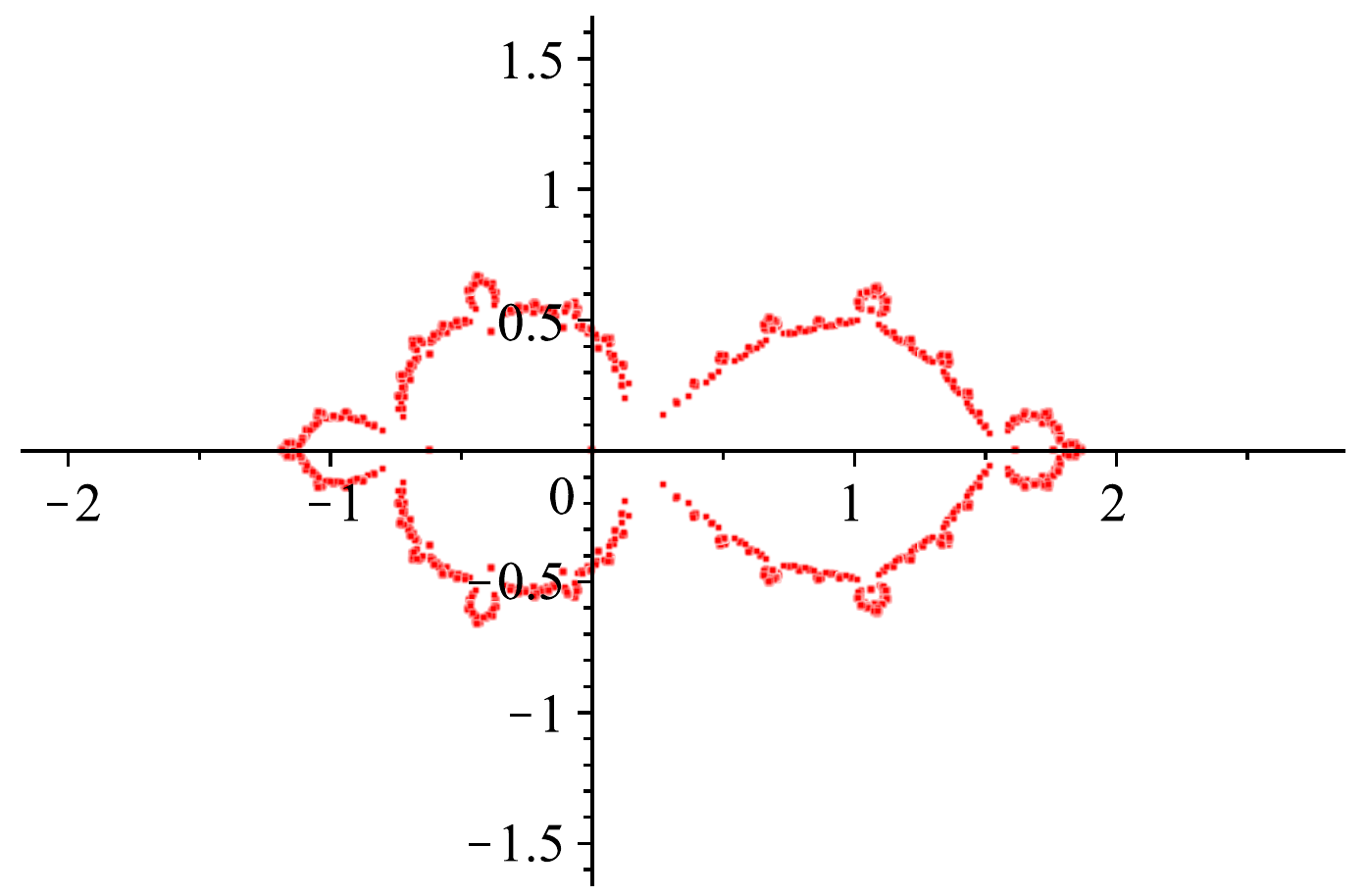}
    \caption{Plot of the two-terminal roots of the sixth repeated gadget replacement of $C_3$ (with two adjacent terminals).}
    \label{fig:C36iterations}
\end{figure}

It seems unclear whether the two-terminal attractor (which is a Julia set) is connected or not -- does it connect up to the root at $z = 0$? We can use the conditions mentioned earlier to analyze the connectivity. The two-terminal reliability here is obviously $f = p + p^2 - p^3$. Its critical points are the solutions to $f^\prime = 1 + 2p - 3p^2 = 0$. The latter has roots at $p = 1$ and $-1/3$. Clearly $f(1) = 1$, so the forward orbit of $1$ under the two-terminal reliability function is bounded.  What about the forward orbit of $-1/3$? It is easy to check that $f$ has a local minimum of $-5/27$ at $p = 1/3$, and that $f([-1,1]) = [-5/27,1] \subset [-1,1]$. It follows that the forward orbit of $-1/3$ is also bounded, and we conclude that the two-terminal attractor is connected.

The argument for $C_4$ with adjacent terminals is a bit more involved (and it is less certain from Figure~\ref{fig:C4-adjacentTerminalsRoots-Fractal} whether it is connected or not).  The two-terminal reliability of $C_4$ with adjacent terminals is $f = p + p^3 -p^4$, which has critical points $1$, $r_1 = \displaystyle{\frac{-1}{8} + \frac{\sqrt{15}}{8}i}$ and  $r_2 = \displaystyle{\frac{-1}{8} + \frac{\sqrt{15}}{8}i}$. Clearly the forward orbit of $1$ is bounded (it is a fixed point of $f$). To determine whether the forward orbits of the others are bounded, a simple calculation determines that $f^2(r_1))$ (and $f^2(r_2))$ both have modulus $\sqrt{388912639}/65536 < 0.31$. Moreover, if we substitute $p = re^{i\theta}$ into $f$, we find that for $r \leq 1/3$,
\begin{eqnarray}
|f(re^{i\theta})| & = & r^2 \left( -8\cos^3 \theta \cdot r^3 + 4 \cos^2 \theta \cdot r^2+ \cos \theta \cdot (-2r^5+6 r^3) 
\right. \nonumber \\ 
 & & \left. + r^6+r^4-2r^2+1 \right) \nonumber\\ 
 & \leq & r^2 \left( 8r^3 + 4r^2 + (-2r^5 + 6r^3) + r^6
+r^4+2r^2 + 1 \right) \nonumber\\ 
 & = & (r^6-2r^5+r^4+14r^3+6r^2+1)r^2 \label{r}
 \end{eqnarray}
 The function of $r$ in (\ref{r})  is increasing on $[0,1/3]$, with a maximum less than $1/3$, so it follows that the forward orbit of any point with modulus at most $1/3$ (which includes $f^2(r_1)$ and 
$f^2(r_2)$) is bounded, and hence the forward orbits of the critical points $r_1$ and $r_2$ are bounded, and the fractal is connected.

One more example --  let's consider the Julia set for the two-terminal reliability of even cycles $C_{2k}$ with antipodal terminals $s$ and $t$ (i.e. points at maximum distance). The two terminal reliability is $f = 2p^k - p^2k$. As 
\[ f^\prime = 2k p^{k-1} - 2kp^{2k-1} = 2kp^{k-1} \left( 1 - p^{k}\right),\]
the critical points are $0$ and the $k$-th roots of unity. Clearly the forward orbit of $0$ is bounded as it is a fixed point. However, if $\omega$ is any $k$-th root of unity, then 
\[ f(\omega) = 2\omega^k - \omega^{2k} = 2-1 = 1,\]
and as $1$ is a fixed point, the orbits of any such critical point $\omega$ is bounded. Thus the Julia set for the associated two-terminal reliability polynomial is connected  (see Figure~\ref{fig:evencycles-antipodalTerminalsRoots-Fractal}).

\begin{figure}[h!]
    \centering
    \includegraphics[width=3.65in]{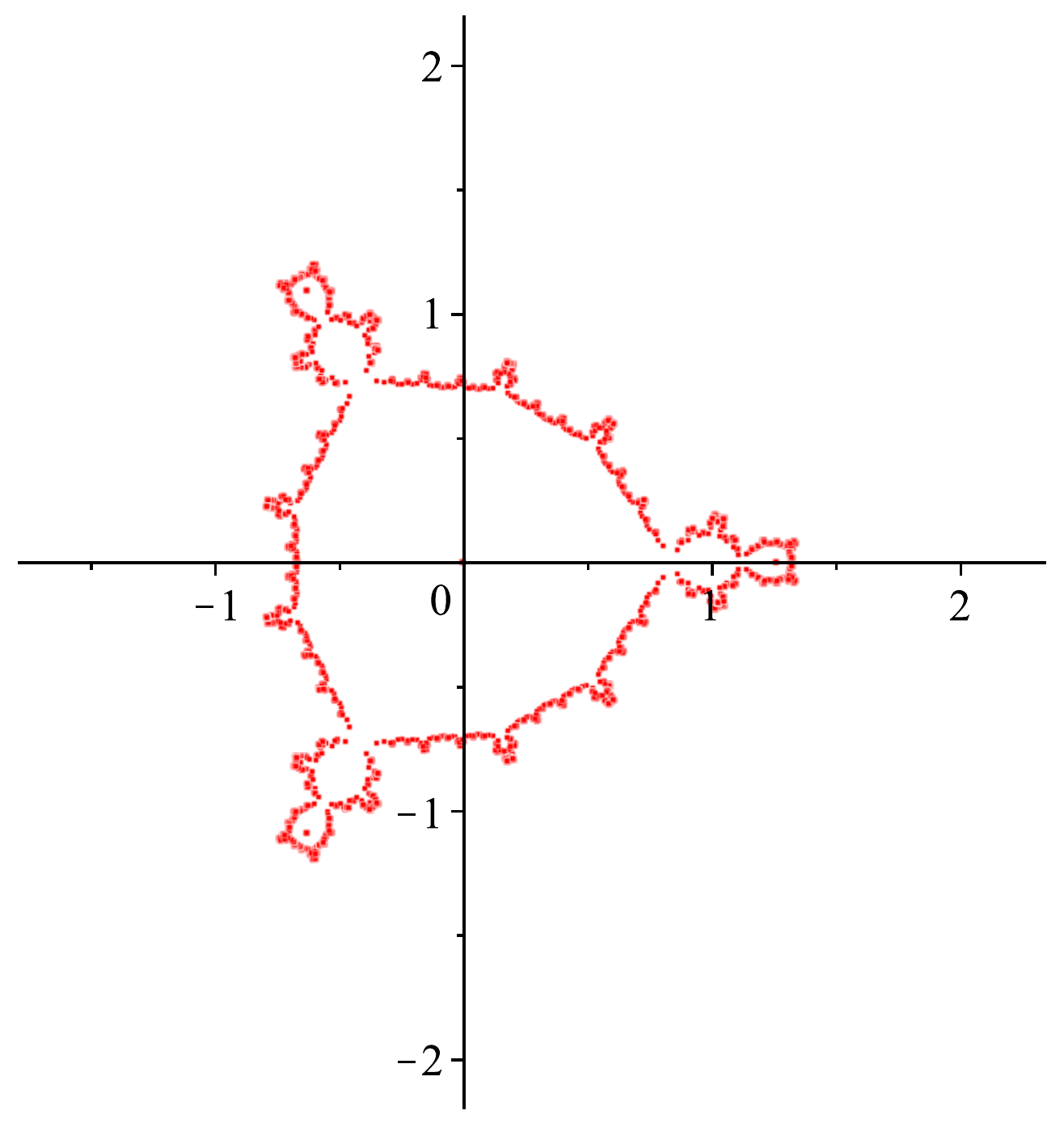}
    \caption{Plot of the two-terminal roots of the fourth repeated gadget replacement of $C_6$ with two antipodal terminals.}
     \label{fig:evencycles-antipodalTerminalsRoots-Fractal}
\end{figure}

One might wonder if the Julia set of every two-terminal reliability polynomial is connected, but the answer is no -- a plot of an approximation to the two-terminal attractor of $K_{2,3}$ (the complete bipartite with cells of cardinality $2$ and $3$), with adjacent terminals $s$ and $t$, is shown in Figure~\ref{fig:evencycles-antipodalTerminalsRoots-Fractal}. Here $f = \TRelp{K_{2,3}}{s}{t}  = p^6-p^5-2p^4+2p^3+p$. The critical points include $r \approx -1.157582493$, and $f^{3}(r) \approx 1.4 \cdot 10^{14}$. It is not hard to check that $|f(z)| > |z|^6-|z|^5-2|z|^4-2|z|^3-|z| \geq |z|^2$ for $|z| > 2.4$, and $|f(r)| > 3.3$, so it follows that the iterates of $f$ at $r$ are unbounded, and hence the associated Julia set (which is the two-terminal attractor of $K_{2,3}$) is disconnected (see Figure~\ref{fig:K23TerminalsRoots-Fractal}).

\begin{figure}[h!]
    \centering
    \includegraphics[width=3.65in]{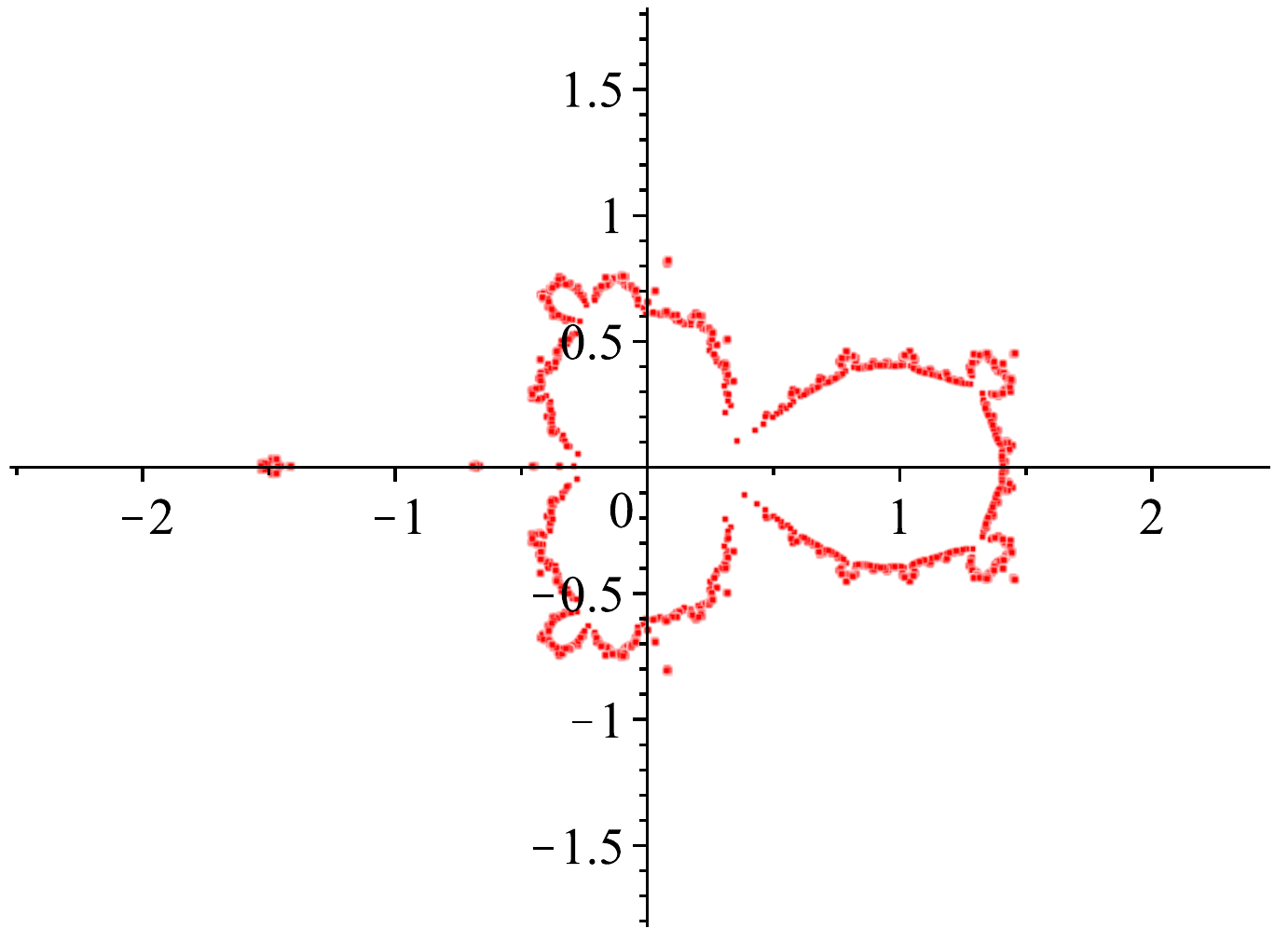}
    \caption{Plot of the two-terminal roots of the fourth repeated gadget replacement of $K_{2,3}$ with two adjacent terminals.}
     \label{fig:K23TerminalsRoots-Fractal}
\end{figure}

\section{Conclusion}

We have seen that, while all-terminal and two-terminal reliabilities have many similarities, their roots behave remarkably differently. We leave with a few easily stated (but intriguing!) open problems.

\begin{prob} 
What is the closure of the {\em real} roots of two terminal reliability polynomials?
\end{prob}

We know that in the all-terminal case, the closure of the real roots is precisely  $\{0\} \cup [1,2]$ (see \cite{brownColbournRoots}).  However, what can be said about two-terminal roots? Unfortunately, we are unable to answer the question.

\begin{prob} 
Are two-terminal roots bounded?
\end{prob}

The two-terminal roots are not as closely related to the disk $|z - 1| \leq 1$ as all-terminal roots seem to be, but we do not have any two-terminal roots of large modulus yet. Can some be located, or is there a large enough disk that contains all two-terminal roots?

\begin{prob} 
When is the two-terminal attractor (or Julia set of the two-terminal reliability, if the terminals are nonadjacent) connected?
\end{prob}

Many of the two-terminal attractors/Julia sets we found were indeed connected, but not all. Are most connected? The problem is likely to be quite difficult, as one needs not only to find the critical points but determine if their forward orbits are bounded. We remark that if one can find a two-terminal reliability polynomial {\em all} of whose critical points have unbounded orbits, then the associated Julia set will indeed be \textbf{totally disconnected} (that is, it will contain nontrivial connected subsets). This seems unlikely, as for most graphs, the two-terminal reliability is flat at $p = 1$, which means that $1$ will be a bounded critical point.

\bibliographystyle{plain}
\bibliography{simple}

\end{document}